\documentclass[11pt]{amsart}

\usepackage[a4paper,margin=1.15in]{geometry}
\usepackage{amsmath,amssymb,amsthm,amsfonts}
\usepackage{mathrsfs}
\usepackage{mathtools}
\usepackage{enumitem}
\usepackage{hyperref}
\usepackage[nameinlink,capitalise]{cleveref}
\usepackage{array}
\hypersetup{
    colorlinks=true,
    linkcolor=blue,
    citecolor=blue,
    urlcolor=blue
}

\theoremstyle{plain}
\newtheorem{theorem}{Theorem}[section]
\newtheorem{proposition}[theorem]{Proposition}

\theoremstyle{definition}
\newtheorem{definition}[theorem]{Definition}
\newtheorem{example}[theorem]{Example}

\theoremstyle{remark}

\title[Curvature on Some Kähler Toric Manifolds]{Curvature on Some Kähler Toric Manifolds}

\author{Xingluan Wang}
\thanks{E-mail: wxlmath@mail.ustc.edu.cn. School of Mathematical Sciences,
University of Science and Technology of China, Hefei 230026, P. R. China.}

\begin{document}

\begin{abstract}
We use the Guillemin--Abreu formalism to give a unified treatment of Kähler metrics with positive holomorphic bisectional curvature on $\mathbb C^n$ and positive holomorphic sectional curvature on $\mathcal O(-\ell)$ and the Hirzebruch manifolds $M_{n,\ell}$.


\end{abstract}

\maketitle

\section{Introduction}
\subsection{Background and Main Results}

The construction of Kähler metrics satisfying prescribed curvature conditions is one of the central problems in complex differential geometry. Classical examples include Kähler--Einstein metrics, constant scalar curvature Kähler metrics, extremal Kähler metrics, and positive holomorphic (bi)sectional curvature. In general, the relevant equations are highly nonlinear and depend delicately on the global geometry of the underlying manifold. Obtaining explicit expressions for such metrics is difficult due to the complexity of the curvature conditions. Many constructions proceed by imposing sufficient symmetry so that the curvature equation, originally a nonlinear partial differential equation on the manifold, reduces to an ordinary differential equation for a single profile function satisfying suitable conditions.

One classical approach was introduced by Klembeck \cite{Klembeck}, who considered radially symmetric metrics on $\mathbb{C}^{n}$, and gave conditions under which such metrics are complete and possess positive holomorphic sectional curvature. Building on this radial symmetry, Wu and Zheng \cite{zheng} later simplified the curvature criteria by reducing the verification to a representative point. Cho and Choi \cite{new} gave examples of Bergman metrics on \(\mathbb{C}^{n}\) induced by weighted Fock spaces that possess positive holomorphic sectional curvature.

Another important method comes from the Calabi ansatz \cite{Calabi}. Koiso and Sakane \cite{koiso} used the momentum map as a coordinate to construct Kähler--Einstein metrics on certain projective bundles. Hwang and Singer \cite{andrew} formulated the momentum construction for circle-invariant Kähler metrics. The theory of Hamiltonian \(2\)-forms developed by Apostolov, Calderbank, Gauduchon, and Tønnesen-Friedman \cite{ACGT1,ACGT2,ACGT3,ACGT4} provides a natural extension of the Calabi ansatz. The Calabi ansatz is a special case of a local Kähler metric admitting a Hamiltonian \(2\)-form of order one. These methods have been widely applied to the study of Kähler metrics, including Kähler--Einstein metrics, Kähler--Ricci solitons, constant scalar curvature Kähler metrics, and extremal Kähler metrics; see, for example, \cite{Cayi1,Cayi2,Cayi4,A1,LeBrun,Simanca,yang,duan}.

Our approach is based primarily on the Delzant--Guillemin--Abreu framework for compact symplectic toric manifolds \cite{delzant,insym,gui}. This framework has also been extended to non-compact toric manifolds \cite{non,abreu1,sca2} and to toric orbifolds \cite{abreu2,orbi1,orbi2}. In this framework, the problem of constructing Kähler metrics with prescribed curvature can be reduced to the construction of a suitable symplectic potential.

Abreu \cite{abreu1} applied this framework to Calabi's four-parameter family of \(U(n)\)-invariant extremal Kähler metrics on Hirzebruch manifolds and showed that this family recovers several classical cohomogeneity-one examples in \cite{Calabi,LeBrun,Simanca,Pedersen}. In a similar spirit, Gauduchon used the momentum profile on the total space of
\(\mathcal O(-\ell)\to \mathbb{CP}^{n-1}\) to give a unified treatment of invariant scalar-flat Kähler metrics, organizing several previously known complete scalar-flat examples within a single construction \cite{Paul}. 
For further applications, see \cite{abreu, ALE, abreu2, raza, sca2,orbi2}.

The purpose of this paper is to give a unified formulation of holomorphic
sectional and bisectional curvature on \(\mathbb C^n\), \(\mathcal O(-\ell)\),
and \(M_{n,\ell}=\mathbb P(\mathcal O(-\ell)\oplus \mathbf 1_{\mathbb{CP}^{n-1}})\), and to recover the corresponding positivity results through a simpler argument than those in the existing literature. In this respect, our approach is inspired by Gauduchon's unified treatment of invariant scalar-flat Kähler metrics via momentum profiles and symplectic potentials \cite{Paul}, while here the focus is on holomorphic sectional and bisectional curvature.

We show that Kähler toric metrics on \(\mathbb{C}^{n}\), \(\mathcal{O}(-\ell)\), and \(M_{n,\ell}\) are described uniformly by the following ansatz, with the case of \(\mathbb C^n\) obtained by setting \(\ell=1\) and the three cases distinguished by different conditions imposed on \(\Theta(x)\):

\begin{align}G(x)=\frac{1}{2}\sum^{n}_{i=1}\ell x_{i}\log(\ell x_{i})-\frac{1}{2}\ell x\log(\ell x)+F(x), \textnormal{ where } F''(x)=\frac{1}{\Theta(x)}. \end{align}

\begin{theorem}
For $\mathbb C^n$, suppose that $\Theta(x)$ is positive  on $(0,+\infty)$, and extends smoothly to $x=0$ with \begin{equation}\Theta(0)=0, ~\Theta'(0)=2,~ \Theta''(x) < 0~(x \geq 0).\end{equation} Then  the corresponding symplectic potential yields a complete Kähler metric with positive holomorphic bisectional curvature on $\mathbb C^n$.
\end{theorem}

The above theorem parallels the positivity criteria obtained by Duan--Guan and by Wu--Zheng for \(U(n)\)-invariant K\"ahler metrics on \(\mathbb C^n\); see \cite[Corollary~1]{duan} and \cite[Theorem~2]{zheng}.
\begin{theorem}
Let $\lambda>0,~ \varepsilon \in (0,2)$. For $\mathcal{O}(- \ell)$, suppose that $\Theta(x)$ is positive  on $(\lambda,+\infty)$, and extends smoothly to $x=\lambda$ with \begin{equation}\Theta(\lambda)=0,~ \Theta'(\lambda)=2,~ \Theta''(x)<0~(x \geq \lambda).\end{equation} If, in addition, the following holds:
\begin{align} 
(\lambda - 2x) \Theta''(x) -2  \Theta'(x)>0   \text{ when }  \ell=1,\end{align}\begin{align} 
 \varepsilon(\lambda - 2x) \Theta''(x)-4 \ell  \Theta'(x) >0  \text{ and } \Theta(x)<\frac{2-\varepsilon}{\ell-1}\lambda  \text{ when }   \ell>1, \end{align}
then the corresponding symplectic potential yields a complete Kähler metric on $\mathcal{O}(- \ell)$ with positive holomorphic sectional curvature.
\end{theorem}

The above theorem recovers positivity results \cite{duan} in the language of the Guillemin--Abreu formalism.
\begin{theorem}
In every Kähler class on \(M_{n,\ell}\), the canonical momentum profile
\[\Theta(x)
=
\frac{2x(x-a)(b-x)}
{\ell(x-a)(b-x)+x(b-a)}\] defines a complete Kähler metric with positive holomorphic sectional curvature on \(M_{n,\ell}\), where the parameters $b>a>0$ determine the Kähler class.
\end{theorem}

The above theorem gives a simpler example than the metrics constructed by Yang and Zheng \cite[Theorem 1.3]{yang}.

\subsection{Organization of the Paper}

In Section 2, we recall the aspects of Guillemin–Abreu theory and explain how toric Kähler metrics with prescribed curvature can be constructed via suitable symplectic potentials.

 In Section 3, we derive the admissibility conditions for the symplectic potentials on \(\mathbb C^n\), \(\mathcal O(-\ell)\), and \(M_{n,\ell}\), and compute the corresponding curvature expressions.

In Section~4, we apply these formulas to \(\mathbb C^n\), \(\mathcal O(-\ell)\), and \(M_{n,\ell}\). We recover the known positivity criteria for \(\mathbb C^n\) and \(\mathcal O(-\ell)\), and show that the canonical toric metric on \(M_{n,\ell}\) has positive holomorphic sectional curvature in every Kähler class.

\section{Guillemin-Abreu Theory}\label{sec2}

\begin{definition}({\cite[Definition 18]{abreu1}})
A toric symplectic manifold is a connected symplectic manifold
\((M^{2n},\omega)\), equipped with an effective Hamiltonian action of the
\(n\)-torus
$
\tau:\mathbb{T}^n \cong \mathbb{R}^n/2\pi\mathbb{Z}^n
\hookrightarrow \operatorname{Ham}(M,\omega),
$
such that the corresponding moment map
$
\mu:M\to \mathbb{R}^n,
$
which is well-defined up to a constant, is proper onto its convex image
$
P=\mu(M)\subset \mathbb{R}^n.
$ Note that the requirement that the moment map be “proper onto its convex
image” is automatic for compact manifolds.
\end{definition}

\begin{definition}({\cite[Definition 22]{abreu1}})
A convex polyhedral set \(P\) in \(\mathbb{R}^n\) is called simple and integral if
\begin{enumerate}
\item[(i)]  there are \(n\) edges meeting at each vertex \(p\).

    \item[(ii)]  the edges meeting at the vertex \(p\) are rational, i.e., each edge is of the form
$
    p+t v_i, 0\leq t\leq \infty,
$
    where \(v_i\in \mathbb{Z}^n\).

    \item[(iii)]  the \(v_1,\ldots,v_n\) in \(\textnormal{(ii)}\) can be chosen to be a
    \(\mathbb{Z}\)-basis of the lattice \(\mathbb{Z}^n\).
\end{enumerate}

A facet is a face of \(P\) of codimension one.

A Delzant set is a simple and integral convex polyhedral set
\(P\subset \mathbb{R}^n\). A Delzant polytope is a compact Delzant set.

Two Delzant sets are said to be isomorphic if one can be mapped to the other
by a translation.
\end{definition}

If $(M,\omega_M,\tau_M)$ is a toric symplectic manifold with moment map
\(\mu:M\to\mathbb R^n\), then \(P\equiv\mu(M)\) is a Delzant set \cite[Theorem 23]{abreu1}. Each Delzant set \(P\) is associated,
via an explicit symplectic reduction construction, with a canonical K\"ahler toric
manifold
\[
(M_P,\omega_P,\tau_P,\mu_P,J_P)
\quad \text{such that} \quad
\mu_P(M_P)=P.
\]
Moreover, the image \(P=\mu(M)\) determines \((M,\omega_M,\mu_M)\) up to isomorphism \cite[Proposition 6.5]{non}.

\begin{theorem}[{\cite[Theorem 2.8]{insym}}]\label{1}
Let $(M_P, \omega_P, \tau_P)$ be the toric symplectic manifold associated to a Delzant polytope $P=\{ l_i \geq 0,i=1,\cdots,d\} \subset \mathbb{R}^n$, and $J$ any compatible toric complex structure. Then $J$ is determined by a ``potential'' $G \in C^\infty(P^\circ)$ of the form
\begin{equation}
\label{eq:2.9}
G = G^{\mathrm{can}} + h 
\end{equation}
where $G^{\mathrm{can}}$ is given by $G^{\mathrm{can}}(x)=\frac{1}{2}\sum_{r=1}^{d}l_{r}(x)\log l_{r}(x)$, $h$ is smooth on the whole $P$, and the matrix $(G_{ij}) = \mathrm{Hess}_x(G)$ is positive definite on $P^\circ$ and has determinant of the form
\begin{align}
\det(G_{ij}) = \left[ \delta(x) \cdot \prod_{r=1}^d l_r(x) \right]^{-1}
\end{align}
with $\delta$ being a smooth and strictly positive function on the whole $P$.

\smallskip
Conversely, any such $G$ determines a compatible toric complex structure $J$ on $(M_P, \omega_P)$, which in the $(x,y)$ symplectic coordinates of $M_P^\circ \cong P^\circ \times \mathbb{T}^n$ has the form 
\begin{equation}
J = \begin{bmatrix}
0  & -(G_{ij})^{-1} \\
(G_{ij}) & 0
\end{bmatrix}.
\end{equation}
\end{theorem}

In the non-compact case, an analogous statement holds {\cite[Theorem 39]{abreu1}}. Complete toric compatible complex structures are described by symplectic potentials of the form
$
G = G^{\mathrm{can}} + h ,
$
satisfying the same conditions as in the compact case. Conversely, any such potential defines a toric compatible complex structure, which is not necessarily complete. Here completeness of a toric compatible complex structure is understood in the
following sense.

\begin{definition}[{\cite[Definition 38]{abreu1}}]
A toric
compatible complex structure \(J\) on a symplectic toric manifold $M$ is said to be
complete if the \(J\)-holomorphic vector fields
$
JY_1,\ldots,JY_n
$
are complete, where $
Y_1,\ldots,Y_n
$
are the Hamiltonian vector fields generating the torus action.
\end{definition}

Let $G=G^{\mathrm{can}}+h$ be a candidate symplectic potential. To show that it defines an
admissible toric K\"ahler metric, we shall use the following admissibility criterion, denoted by \((\star)\):
\begin{enumerate}
\item[(i)]  The function \(h\) is smooth on \(P\).

\item[(ii)] The Hessian matrix $(G_{ij})$ is positive
definite on \(P^\circ\).

\item[(iii)] 
The function $\delta=\bigl( \det (G_{ij})\prod_{r=1}^d l_r(x)\bigr)^{-1}$ is smooth and strictly positive on $P$.

\item[(iv)] In the non-compact case, we must also check that the induced toric complex structure is
complete. In the compact case, this condition is automatic.
\end{enumerate}

A key point in the Guillemin--Abreu theory is that the complex structures
appearing above are biholomorphically equivalent. More precisely, if $M$ is compact, then for any toric compatible complex structure \(J\) on \((M_P,\omega_P,\tau_P) \), we have
\begin{equation}(M,J_M,\tau_M) \cong (M_P,J_P,\tau_P)\cong (M_P,J,\tau_P)\end{equation} by equivariant biholomorphisms \cite[Remark 2.4,\ Proposition A.1]{insym}. If $M$ is non-compact, then the analogue of the first identification holds for AK-toric manifolds \cite[Lemma 2.14]{Uni1} and the second identification holds if the toric compatible complex structure \(J\) is complete. The same argument as in Appendix~A of \cite{insym} gives the second identification in the complete non-compact case; see also \cite[Remark 2.11]{sca2}. Note that there is no immediate relation between completeness of a toric compatible complex structure and completeness of the associated toric Kähler metric \cite[Remark 2.12]{sca2}. 

The completeness of a toric
compatible complex structure can be studied as follows.
If \(J\) is determined by \(G\), set
$
\xi=\nabla G:P^\circ\to \mathbb R^n
$, 
then the \(J\)-holomorphic coordinates are given by
$
(\xi,\theta)=\left(\frac{\partial G}{\partial x},\theta\right)
$. 
\(J\) is complete precisely when the map \(\xi\) is surjective, i.e.,\ $\nabla G(P^\circ)=\mathbb R^n$. The above argument is the same as in the proof of \cite[Proposition 5.4]{sca2}.

We illustrate the verification of the \(J\)-completeness condition in the case of \(\mathbb C^n\); the case of \(\mathcal O(-\ell)\) is analogous.

\begin{example}
Consider the standard toric Kähler manifold \(\mathbb C^n\), whose Delzant set is $
P=\mathbb R_{\geq 0}^n$.
Here \(x\) denotes $\sum_{i=1}^n x_i $.
Let
\begin{equation}
G(x)
=
\frac12\sum_{i=1}^n x_i\log x_i
-\frac12 x\log x
+F(x) .
\end{equation}

For \(y=(y_1,\ldots,y_n)\in\mathbb R^n\), the inverse of \(\nabla G\), whenever
defined, is given by
\begin{equation}
x_i
=
(F')^{-1}\left(
\frac12\log\sum_{j=1}^n e^{2y_j}
\right)
\frac{e^{2y_i}}{\sum_{j=1}^n e^{2y_j}} .
\end{equation}
Indeed, the function $\frac12\log \sum_{j=1}^n e^{2y_j}$ ranges over all of $\mathbb R$: it takes the value $\frac12\log(n e^{2c}) $ when \(y_1=\cdots=y_n=c\). Hence \(\nabla G(P^\circ)=\mathbb R^n\) is equivalent to
$F'((0,+\infty))=\mathbb R$, i.e.,\ 
\begin{equation}
\int_0^1 \Theta(x)^{-1}\,dx=+\infty,
\qquad
\int_1^{+\infty}\Theta(x)^{-1}\,dx=+\infty .
\end{equation}
\end{example}

\section{Curvature Formulas for the Three Model Spaces
}\label{sec3}

In this section we derive explicit formulas for the holomorphic (bi)sectional curvature of \(\mathbb{C}^n\), $\mathcal O(-\ell)$, and $M_{n,\ell}$, and give criteria for positivity. The point is that, in each case, after a suitable choice of $h$, the symplectic potential $G=G^{\mathrm{can}}+h$ can be written in the same form. For this ansatz, checking \((\star)\) for \(h\) reduces to checking the corresponding conditions on \(\Theta(x)\).

For the following three model spaces, we use the following common notation. Set $x=\sum_{i=1}^n x_i $ and $\eta(\cdot)=\frac{1}{2}(\cdot)\log (\cdot)$. Throughout this section, assume \(b>a>0,\ \ell \in \mathbb{Z}_{>0}\) and $\lambda>0$. Following the method in \cite[Example 42]{abreu1}, up to $GL(n,\mathbb R)$ transformations, the corresponding Delzant sets of \(\mathbb{C}^{n}\), \(\mathcal O(-\ell)\), and \(M_{n,\ell}\) are
\begin{equation}
P_1=\{x_i\geq 0\},\ P_2=\{\ell x_i \geq 0,~x-\lambda \geq 0  \},\ P_3= \{ \ell x_i\geq 0, x-a\geq 0, b-x\geq 0  \},
\end{equation}
respectively, where the indexed inequalities are understood to hold for \(1\leq i \leq n\).  We then define
\begin{equation}
h_1=-\eta(x)+F(x),\ h_2=-\eta(\ell x)-\eta(x-\lambda)+F(x),\ h_3=-\eta(\ell x)-\eta(x-a)-\eta(b-x)+F(x).
\end{equation}

For the three Delzant sets under consideration, the admissible symplectic potentials
$
G_i=G_i^{\mathrm{can}}+h_i
$
have the same form, with the case of \(\mathbb C^n\) obtained by setting \(\ell=1\).  
\begin{align}G(x)=\frac{1}{2}\sum^{n}_{i=1}\ell x_{i}\log(\ell x_{i})-\frac{1}{2}\ell x\log(\ell x)+F(x).  \end{align}

A direct computation gives $\det(G_{ij})=\ell^{n-1} x(2^{n-1}\Theta(x)\prod{x_{i}})^{-1}$. Moreover $(G_{ij})$ is positive definite on the interior if and only if $\Theta(x)>0$. By \cite[Equation~(2.2)]{duan}, \cite[Equation~(2)]{zheng} and \((\star)\), setting $\Theta(x)>0$ on the interval, we have

\begin{center}
\small
\setlength{\tabcolsep}{5pt}
\renewcommand{\arraystretch}{1.15}
\begin{tabular}{@{}l l l p{0.25\textwidth} p{0.25\textwidth}@{}}
\hline
Space & Interval & Endpoint conditions & $J$-completeness & $g$-completeness\\
\hline
\(\mathbb C^n\)
&
\((0,+\infty)\)
&
\(\Theta(0)=0,\ \Theta'(0)=2\)
&
\begin{tabular}[t]{@{}l@{}}
\(\int_0^1\Theta(x)^{-1}\,dx=+\infty\),\\
\(\int_1^{+\infty}\Theta(x)^{-1}\,dx=+\infty\)
\end{tabular}
&
$
\int_{1}^{+\infty}\Theta(x)^{-1/2}\,dx=+\infty
$

\\[0.45em]

\(\mathcal O(-\ell)\)
&
\((\lambda,+\infty)\)
&
\(\Theta(\lambda)=0,\ \Theta'(\lambda)=2\)
&
\begin{tabular}[t]{@{}l@{}}
\(\int_\lambda^{\lambda+1}\Theta(x)^{-1}\,dx=+\infty\),\\
\(\int_{\lambda+1}^{+\infty}\Theta(x)^{-1}\,dx=+\infty\)
\end{tabular}
&
$
\int_{\lambda+1}^{+\infty}\Theta(x)^{-1/2}\,dx=+\infty
$
\\[0.45em]

\(M_{n,\ell}\)
&
\((a,b)\)
&\begin{tabular}[t]{@{}l@{}}
\(\Theta(a)=\Theta(b)=0,\)\\ \( \Theta'(a)=2,\ \Theta'(b)=-2\)\end{tabular}
&
Automatic
&
Automatic
\\
\hline
\end{tabular}
\end{center}

The conditions listed in the table agree exactly with the corresponding endpoint conditions for the momentum profile in \cite[pp.~270, 272, and 285]{Paul}.
The conditions \(\Theta(a)=0\) and \(\Theta'(a)=2\) imply that, near \(x=a\) and on the side \(x>a\),
\begin{equation}
F''(x)=\frac{1}{2(x-a)}+\varphi(x), \textnormal{ i.e., } F(x)=\frac{1}{2}(x-a)\log(x-a)+\psi(x),
\end{equation}
where \(\varphi\) and $\psi$ are smooth near \(x=a\).
Thus the boundary conditions have a natural
interpretation in the Guillemin--Abreu theory: they mean that adding facets to the polytope is compatible with the standard potential.

For the \(U(n)\)-invariant metrics on \(\mathbb C^n\), \(\mathcal O(-\ell)\), and \(M_{n,\ell}\), global positivity of the holomorphic (bi)sectional curvature reduces to positivity at a representative point. 
The resulting formulas, stated below, are consistent with \cite[Lemma~1]{duan} and \cite[Proposition~3.6]{yang}. Notice that the case \(\ell=1\) recovers the \(\mathbb C^n\) case.

\begin{proposition}\label{k=2}
At the representative point \((x_1,0,\ldots,0)\), the nonzero curvature components in the normalized frame associated with $G(x)=\frac{1}{2}\sum^{n}_{i=1}\ell x_{i}\log(\ell x_{i})-\frac{1}{2}\ell x \log(\ell x)+F(x)$ are:
\begin{equation}
\begin{gathered}
A=R_{1 \bar{1} 1 \bar{1}}=- \Theta''(x),\ 
B=R_{i \bar{i} 1 \bar{1}}=R_{i \bar{1} 1 \bar{i}}=R_{1 \bar{i} i \bar{1}}=R_{1 \bar{1} i \bar{i}}=\frac{\Theta(x)-x\Theta'(x)}{x^2},\ i \geq 2,\\[1ex]
C=R_{i \bar{i} i \bar{i}}= \frac{-2\ell  \Theta(x)+4x}{\ell x^2},\
\frac{C}{2}=R_{i \bar{j} j \bar{i}}=R_{i \bar{i} j \bar{j}}= \frac{-\ell \Theta(x)+2x}{\ell  x^2},\ 2\leq i,j\leq n,\ i\neq j.\end{gathered}\end{equation}
\end{proposition}

\begin{proof}
We use the results in \cite[Equations~(3)--(4)]{zheng}.
Set $\widetilde{r}=r^2=|z|^2$, and $t=\ell \log r$, then \(\frac{d}{d\widetilde r}
=
\frac{\ell}{2 \widetilde r}\frac{d}{dt}
=
\frac{\ell \Theta}{2\widetilde r}\frac{d}{dx},\) and get
\begin{equation}p(\widetilde r)=\phi(t), \ x=\phi_t,\ \Theta=\phi_{tt},\ f=p'=\frac{\ell x}{2\widetilde{r}},\ h=(\widetilde{r}f)^{'}_{\widetilde{r}}=(\frac{\ell x}{2})^{'}_{\widetilde{r}}=\frac{\ell^2 \Theta}{4\widetilde{r}}.\end{equation}
\begin{equation}\frac{f'}{f}=(\log x)'_{\widetilde{r}}-(\log \widetilde{r})'_{\widetilde{r}}=\frac{\ell  \Theta-2x}{2\widetilde r x},\ \frac{h'}{h}=(\log \Theta)'_{\widetilde{r}}-(\log \widetilde{r})'_{\widetilde{r}}=\frac{\ell \Theta'-2}{2 \widetilde{r}}\end{equation}
\begin{equation}
A
=
-\frac1h\left(\frac{\widetilde r h'}h\right)'
=
-\frac{4\widetilde r}{\ell^2\Theta}
\left(\frac{\ell\Theta'-2}{2}\right)'_{\widetilde r}
=
-\frac{4\widetilde r}{\ell^2\Theta}
\frac{\ell \Theta}{2\widetilde r}\frac{\ell \Theta''}{2}
=
-\Theta''.
\end{equation}
\begin{equation}
B
=
\frac{f'}{f^2}
-
\frac{h'}{hf}
=
\frac{\ell\Theta-2x}{\ell x^2}
-
\frac{\ell\Theta'-2}{\ell x}
=
\frac{\Theta-x\Theta'}{x^2}.
\end{equation}
\begin{equation}C
=
-\frac{2f'}{f^2}
=
-2\cdot
\frac{\ell\Theta-2x}{\ell x^2}
=
\frac{4x-2\ell\Theta}{\ell x^2}.\end{equation}

\end{proof}

For the metrics under consideration, positivity of holomorphic bisectional curvature is equivalent to
$
A>0,B>0$, and $C>0$.
Similarly, positivity of holomorphic sectional curvature is equivalent to
$
A>0,C>0$, and $2B+\sqrt{AC}>0.
$

\section{Examples Revisited in Polytope Coordinates}\label{sec4}

\subsection{Metrics on \(\mathbb C^n\)}

\begin{theorem}\label{t1}
For $\mathbb C^n$, suppose that $\Theta(x)$ is positive  on $(0,+\infty)$, and extends smoothly to $x=0$ with \begin{equation}\Theta(0)=0, ~\Theta'(0)=2,~ \Theta''(x) < 0~(x \geq 0).\end{equation} Then  the corresponding symplectic potential yields a complete Kähler metric with positive holomorphic bisectional curvature on $\mathbb C^n$.
\end{theorem}

\begin{proof}

The inequality \(A>0\) follows from \(A=-\Theta''\). Since \(\Theta''<0\) and \(\Theta'(0)=2\), we have \(\Theta(x)>x\Theta'(x)\), which gives \(B>0\). The inequality \(\Theta(x)<2x\) gives \(C>0\), and it also implies \(\int_1^\infty \Theta(x)^{-1}\,dx=\infty\) and \(\int_1^\infty \Theta(x)^{-1/2}\,dx=\infty\). Finally, \(\Theta(x)\sim 2x\) near $x=0$ gives \(\int_0^1 \Theta^{-1}dx=+\infty\).

\end{proof}

The above theorem parallels the positivity criteria obtained by Duan--Guan and by Wu--Zheng for \(U(n)\)-invariant K\"ahler metrics on \(\mathbb C^n\); see \cite[Corollary~1]{duan} and \cite[Theorem~2]{zheng}.

\subsection{Metrics on $\mathcal{O}(-\ell)$}

We first observe that no metric within the present ansatz on
\(\mathcal O(-\ell)\) has positive
holomorphic bisectional curvature. Indeed, positive holomorphic bisectional curvature requires
$\Theta(x)-x\Theta'(x)>0.$
Taking the limit \(x\to\lambda^+\) gives
$
0\leq \Theta(\lambda)-\lambda\Theta'(\lambda)=-2\lambda.
$ This contradicts \(\lambda>0\).

We now give sufficient conditions for the metric to have positive holomorphic sectional curvature.
We reformulate the method of Duan and Guan \cite[pp.~87--89]{duan} in action-angle coordinates. It suffices to study the positivity of \(R_{X\bar X X\bar X}\), where \(X=\sum_i \xi_i e_i\) is a unit vector in the normalized frame. 

Set $\tau=|\xi_{1}|^2$. Then
\begin{equation}
\begin{aligned}
R_{X\bar{X}X\bar{X}}
&=A\tau^2+4B\tau(1-\tau)+C(1-\tau)^2 \\
&=-\tau^2\Theta''(x)+\frac{4(\Theta(x)-x\Theta'(x))}{x^2}\tau(1-\tau)+\frac{-2 \ell \Theta(x)+4x}{\ell x^2}(1-\tau)^2.
\end{aligned}\end{equation}

\begin{theorem}
Let $\lambda>0,~ \varepsilon \in (0,2)$. For $\mathcal{O}(- \ell)$, suppose that $\Theta(x)$ is positive  on $(\lambda,+\infty)$, and extends smoothly to $x=\lambda$ with \begin{equation}\Theta(\lambda)=0,~ \Theta'(\lambda)=2,~ \Theta''(x)<0~(x \geq \lambda).\end{equation} If, in addition, the following holds:
\begin{eqnarray} \label{eq:duan1} 
(\lambda - 2x) \Theta''(x) -2  \Theta'(x)>0   \text{ when }  \ell=1,\end{eqnarray}\begin{eqnarray} \label{eq:duan2} 
 \varepsilon(\lambda - 2x) \Theta''(x)-4 \ell  \Theta'(x) >0  \text{ and } \Theta(x)<\frac{2-\varepsilon}{\ell-1}\lambda  \text{ when }   \ell>1, \end{eqnarray}
then the corresponding symplectic potential yields a complete Kähler metric on $\mathcal{O}(- \ell)$ with positive holomorphic sectional curvature.
\end{theorem}

\begin{proof}
Set $\Theta_0(x)=\Theta(x+\lambda)$, then $\Theta_0(x)$ is positive on $(0,+ \infty)$, satisfies $\Theta_0''(x)<0$ for \(x\ge0\), and $\Theta_0(0)=0$, $\Theta_0'(0)=2$. By the argument in the proof of Theorem~\ref{t1} applied to \(\Theta_0\), we have
\begin{equation}-\tau^2\Theta_0''(x)+\frac{4(\Theta_0(x)-x\Theta_0'(x))}{x^2}\tau(1-\tau)+\frac{-2\Theta_0(x)+4x}{x^2}(1-\tau)^2>0 .\end{equation}

Note that $\Theta(x)=\Theta_0(x-\lambda)$. Then, for \(x\ge\lambda\), \begin{equation}-\tau^2\Theta''(x)+\frac{4(\Theta(x)-(x-\lambda)\Theta'(x))}{(x-\lambda)^2}\tau(1-\tau)+\frac{-2\Theta(x)+4(x-\lambda)}{(x-\lambda)^2}(1-\tau)^2>0. \end{equation}

In particular,
\begin{equation}
\begin{gathered}R_1=-\tau^2\Theta''(x)\geq 0, \\
R_2=\frac{4(\Theta(x)-(x-\lambda)\Theta'(x))}{(x-\lambda)^2}\tau(1-\tau)\geq 0, \\
R_3=\frac{-2\Theta(x)+4(x-\lambda)}{(x-\lambda)^2}(1-\tau)^2\geq 0.\end{gathered}\end{equation}

To ensure that $R_{X\bar{X}X\bar{X}}$ is positive, it suffices to prove positivity of the following expression:
\begin{equation}\begin{aligned}
&\ell x^2 R_{X\bar{X}X\bar{X}} - \ell(x - \lambda)^2(R_1 +R_2 )- (x - \lambda)^2  R_3\\
=&2(1 - \ell)(1 - \tau)^2 \Theta(x) 
+ \lambda \left( 
4(1 - \tau)^2 + 4\ell(\tau-1)\tau\, \Theta'(x) \right)+ \lambda \ell \tau^2 (\lambda - 2x) \Theta''(x). \end{aligned}\end{equation}

The constant term of this quadratic polynomial in \(\tau\) is
\(2\bigl(2\lambda-(\ell-1)\Theta(x)\bigr)\), which is positive under the assumptions.

The discriminant condition is
\begin{align}-2 \lambda \ell \left(\Theta'(x)\right)^2 
+ (\lambda - 2x)\left(2\lambda + \Theta(x) - \ell \Theta(x)\right) \Theta''(x) >0.\end{align}

Using the relation $0 \leq \Theta'(x)\leq 2$, we obtain a sufficient condition:
\begin{eqnarray}\Xi=\label{a1} -4 \lambda \ell \Theta'(x) 
+ (\lambda - 2x)\left(2\lambda + \Theta(x) - \ell \Theta(x)\right) \Theta''(x) >0 .\end{eqnarray}

It remains to verify the sufficient condition $\Xi>0$.

When \(\ell=1\), 
$
\Xi=-4\lambda\Theta'(x)+2\lambda(\lambda-2x)\Theta''(x)
=
2\lambda\bigl((\lambda-2x)\Theta''(x)-2\Theta'(x)\bigr)>0.
$
Now assume \(\ell>1\). 
We have
$
2\lambda+\Theta(x)-\ell\Theta(x)
=
2\lambda-(\ell-1)\Theta(x)
>
\varepsilon\lambda.
$
Therefore 
$
\Xi >
-4\lambda\ell\Theta'(x)
+\varepsilon\lambda(\lambda-2x)\Theta''(x)  
=
\lambda\bigl(\varepsilon(\lambda-2x)\Theta''(x)-4\ell\Theta'(x)\bigr)>0. 
$ 

Moreover, since \(\Theta''<0\), \(\Theta(\lambda)=0\), and
\(\Theta'(\lambda)=2\), we have \(\Theta(x)<2(x-\lambda)\).
Therefore
$
\int_{\lambda+1}^{\infty}\Theta(x)^{-1}\,dx=\infty,\
\int_{\lambda+1}^{\infty}\Theta(x)^{-1/2}\,dx=\infty,
$ and \(\Theta(x)\sim 2(x-\lambda)\) near $x=\lambda$ gives \(\int_\lambda^{\lambda+1} \Theta^{-1}dx=+\infty\),
so the associated toric complex structure and the metric are complete.
\end{proof}

\subsection{Metrics on $M_{n,\ell}$}

The preceding method does not apply directly to $M_{n,\ell}$, since the required profile $\Theta(x)$ vanishes at both endpoints of an interval and therefore cannot be obtained by a translation of the profiles $\Theta(x)$ associated with $\mathbb{C}^n$ or $\mathcal{O}(-\ell)$.

By analyzing the positivity inequalities
$
A>0,\; C>0,\; 2B+\sqrt{AC}>0,
$
Yang and Zheng constructed a Kähler metric with positive holomorphic sectional curvature on \(M_{n,\ell}\) \cite[Theorem 1.3]{yang}. We show below that the canonical toric Kähler metric on \(M_{n,\ell}\) arising from the Guillemin–Abreu formalism already has positive holomorphic sectional curvature.

\begin{theorem}
In every Kähler class on \(M_{n,\ell}\), the canonical momentum profile
\begin{equation}\Theta(x)
=
\frac{2x(x-a)(b-x)}
{\ell(x-a)(b-x)+x(b-a)}\end{equation} defines a complete Kähler metric with positive holomorphic sectional curvature on \(M_{n,\ell}\), where the parameters $b>a>0$ determine the Kähler class.
\end{theorem}

\begin{proof}
In the notation used above, the potential for \(M_{n,\ell}\) is
\begin{equation}G(x)=\frac{1}{2}\sum^{n}_{i=1}\ell x_{i}\log(\ell x_{i})
+\frac12(x-a)\log(x-a)
+\frac12(b-x)\log(b-x),
\end{equation}
which is exactly the canonical potential.

Let \(L(x)=\ell(x-a)(b-x)+(b-a)x\) and $N=\ell\bigl(a^2b^2-3abx^2+(a+b)x^3\bigr)+(b-a) x^3$. By the curvature criterion, it is enough to verify
\[
\left\{
\begin{aligned}
A
&=-\Theta''
=\frac{4(b-a)}{L(x)^3}N>0,\\
C
&=\frac{4x-2\ell\Theta}{\ell x^2}
=\frac{4(b-a)}{\ell L(x)}>0,\\
2B+\sqrt{AC}
&=\frac{2\Theta-2x\Theta'}{x^2}+\sqrt{\frac{-\Theta''(4x-2\ell\Theta)}{\ell x^2}}=
\frac{4(b-a)}{L(x)^2}
\left(
x^2-ab+\sqrt{\frac{N}{\ell}}
\right)>0.
\end{aligned}
\right.
\]

Since \(L(x)>0\) for \(a<x<b\), we have  \(C>0\). To prove \(A>0\), set 
$
\phi(x)=a^2b^2-3abx^2+(a+b)x^3
$
, then 
$
\phi'(x)=3x((a+b)x-2ab)
$, 
so $\phi(x)\geq \phi(\frac{2ab}{a+b})=\frac{a^2 b^2(b-a)^2}{(a+b)^2}>0$.
Hence \(A>0\).
It remains to check \(2B+\sqrt{AC}>0\). It is equivalent to
\begin{equation}
x^2-ab+\sqrt{\frac{N}{\ell}}>0.
\end{equation}
This is immediate when $x \geq \sqrt{ab}$. If $a < x< \sqrt{ab}$, then a direct simplification gives,
\begin{equation}
\frac{N}{\ell}-(ab-x^2)^2
=
x^2\left((x-a)(b-x)+\frac{b-a}{\ell}x\right)>0.
\end{equation}
Therefore \(2B+\sqrt{AC}>0\).
\end{proof}

\bibliographystyle{unsrt}
\bibliography{sample}

\end{document}